\documentclass[12pt]{amsart}

\usepackage{tgpagella}
\usepackage{euler}
\usepackage[T1]{fontenc}
\usepackage{amsmath, amssymb}
\usepackage[hidelinks]{hyperref}
\usepackage[english]{babel}
\usepackage{mathrsfs}
\usepackage{eucal}
\usepackage[all]{xy}
\usepackage{tikz}

\newtheorem{thm}{Theorem}[section]
\newtheorem*{thm*}{Theorem}
\newtheorem{lem}[thm]{Lemma}
\newtheorem{fact}[thm]{Fact}
\newtheorem*{prob*}{Problem}

\newtheorem{prop}[thm]{Proposition}
\newtheorem*{prop*}{Proposition}

\newtheorem{cor}[thm]{Corollary}
\newtheorem*{cor*}{Corollary}

\theoremstyle{definition}
\newtheorem{defn}[thm]{Definition}
\newtheorem*{defn*}{Definition}

\newtheorem{remark}[thm]{Remark}

\newtheorem*{question*}{Question}
\newtheorem*{Pquestion*}{Popa's question}

\newtheorem*{conv*}{Convention}
\newtheorem*{thm1*}{Theorem 1}
\newtheorem*{thm2*}{Theorem 2}
\newtheorem*{thmA*}{Theorem A}
\newtheorem*{thmB*}{Theorem B}
\newtheorem*{thmC*}{Theorem C}
\newtheorem*{thmD*}{Theorem D}
\newtheorem*{conj*}{Conjecture}
\newtheorem*{fact*}{Fact}

\newcommand{\dminus}{ 
\buildrel\textstyle\ .\over{\hbox{ 
\vrule height3pt depth0pt width0pt}{\smash-} 
}}

\def\u{\mathsf 1}

\def \tp{\operatorname{tp}}
\def \depth{\operatorname{depth}}

\makeatletter

\def\dotminussym#1#2{%
  \setbox0=\hbox{$\m@th#1-$}%
  \kern.5\wd0%
  \hbox to 0pt{\hss\hbox{$\m@th#1-$}\hss}%
  \raise.6\ht0\hbox to 0pt{\hss$\m@th#1.$\hss}%
  \kern.5\wd0}
\newcommand{\dotminus}{\mathbin{\mathpalette\dotminussym{}}}

\DeclareMathOperator{\tr}{tr}

\def \Th{\operatorname{Th}}
\def \R{\mathcal R}
\def \u{\mathcal U}

\begin{document}

%%%%%%%%%%%%%%%%%%%%%%%%%%%%%%%%%%%%%%%%%%%%%%

\title{Properties expressible in small fragments of the theory of the hyperfinite II$_1$ factor}
\author{Isaac Goldbring and Bradd Hart}

\address{Department of Mathematics\\University of California, Irvine, 340 Rowland Hall (Bldg.\# 400),
Irvine, CA 92697-3875}
\email{isaac@math.uci.edu}
\urladdr{http://www.math.uci.edu/~isaac}

\address{Department of Mathematics and Statistics, McMaster University, 1280 Main St., Hamilton ON, Canada L8S 4K1}
\email{hartb@mcmaster.ca}
\urladdr{http://ms.mcmaster.ca/~bradd/}

\begin{abstract}
We show that any II$_1$ factor that has the same 4-quantifier theory as the hyperfinite II$_1$ factor $\R$ satisfies the conclusion of the Popa Factorial Commutant Embedding Problem (FCEP) and has the Brown property.  These results improve recent results proving the same conclusions under the stronger assumption that the factor is actually elementarily equivalent to $\R$.  In the same spirit, we improve a recent result of the first-named author, who showed that if (1) the amalgamated free product of embeddable factors over a property (T) base is once again embeddable, and (2) $\R$ is an infinitely generic embeddable factor, then the FCEP is true of all property (T) factors.  In this paper, it is shown that item (2) can be weakened to assume that $\R$ has the same 3-quantifier theory as an infinitely generic embeddable factor.    
\end{abstract}

\maketitle

\section{Introduction}

The following problem of Popa is the main motivation for the work in this paper:

\begin{prob*}[Popa's Factorial Commutant Embedding Problem (FCEP)]
Suppose that $M$ is a separable embeddable factor.  Does there exist an embedding $i:M\hookrightarrow \R^\u$ with factorial commutant, that is, such that $i(M)'\cap \R^\u$ is a factor?
\end{prob*}

Until recently, very little progress on the FCEP had been made.  In \cite{jung}, the following theorem was proven:

\begin{thm1*}\label{eeRFCEP}
If $M$ is elementarily equivalent to $\R$, then $M$ satisfies the FCEP.
\end{thm1*}

Recall that II$_1$ factors $M$ and $N$ are elementarily equivalent, denoted $M\equiv N$, if, for any sentence $\sigma$ in the language of tracial von Neumann algebras, one has $\sigma^M=\sigma^N$.  A logic-free definition can be given using the Keisler-Shelah Theorem:  $M$ and $N$ are elementarily equivalent if and only if they have isomorphic ultrapowers.\footnote{If one is willing to assume the continuum hypothesis, this can even be improved by saying that $M$ and $N$ are elementarily equivalent if and only if $M^\u\cong N^\u$ for any nonprincipal ultrafilter on $\mathbb N$.}  By \cite[Theorem 4.3]{MTOA3}, any separable II$_1$ factor $M$ has continuum many nonisomorphic separable II$_1$ factors elementarily equivalent to it, whence Theorem 1 gave continuum many new examples of separable II$_1$ factors satisfying the FCEP.

In this paper, we weaken the assumption of the previous theorem and arrive at the same conclusion.  We say that II$_1$ factors $M$ and $N$ are $k$-elementarily equivalent, denoted $M\equiv_k N$, if they agree on all formulae of quantifier-complexity at most $k$.  (This will be defined precisely in the last section.). The following is an imprecise version of our first main result:

\begin{thmA*}\label{4FCEP}
If $M\equiv_4 \R$, then $M$ satisfies the FCEP.
\end{thmA*}

In another direction, one of the main results of \cite{Popa} was progress on the FCEP problem for embeddable\footnote{In this paper, we use the term \textbf{embeddable} as an abbreviation for $\R^\u$-embeddable.} property (T) factors:

\begin{thm2*}\label{propT}
Suppose that the following two statements are true:
\begin{enumerate}
    \item Whenever $M_1$ and $M_2$ are embeddable II$_1$ factors with a common property (T) subfactor $N$, then the amalgamated free product $M_1*_N M_2$ is also embeddable.
    \item $\R$ is an \emph{infinitely generic} embeddable factor.
\end{enumerate}
Then every embeddable property (T) factor satisfies the FCEP.
\end{thm2*}

Infinitely generic factors form a large class of ``rich'' II$_1$ factors and more information about them can be found in \cite{ecfactor}.  In \cite{ecfactor}, it was claimed that $\R$ is an infinitely generic embeddable factor.  However, the proof there is incredibly flawed and settling the question of whether or not $\R$ is actually an infinitely generic embeddable factor remains an important open question.

Ideally, one would like to remove the model-theoretic assumption (2) in the previous theorem, leaving only the operator-algebraic obstacle (1).  Item (2) in the previous theorem is equivalent to the statement that $\R$ is elementarily equivalent to an infinitely generic embeddable factor.  Consequently, the following theorem, a consequence of a more general result proven in Section 4, is a strengthening of the previous result:

\begin{thmB*}\label{newpropT}
Suppose that the following two statements are true:
\begin{enumerate}
    \item Whenever $M_1$ and $M_2$ are embeddable II$_1$ factors with a common property (T) subfactor $N$, then the amalgamated free product $M_1*_N M_2$ is also embeddable.
    \item[(2')] There is an infinitely generic embeddable factor $M$ such that $M\equiv_3 \R$.
\end{enumerate}
Then every embeddable property (T) factor satisfies the FCEP.
\end{thmB*}

It is worth noting that any infinitely generic embeddable factor $M$ satisfies $M\equiv_2 \R$.  In Section 4, we also note that the statement that there is an infinitely generic embeddable factor $M$ such that $M\equiv_3 \R$ is already known to be ``halfway true.''

% In \cite{jung}, we raised the question whether or not every existentially closed embeddable factor satisfies the FCEP; if this were true, it would give a ``generic'' positive solution to the FCEP.  In this paper, we make some partial progress on this question by elucidating when a large, natural class of e.c. embeddable factors, the so-called infinitely generic embeddable factors, satisfy the FCEP.

A crucial ingredient to the proof of Theorem 1 above is the following result of Nate Brown \cite[Theorem 6.9]{brown}:

\begin{fact*}
If $N$ is a separable subfactor of $\R^\u$, then there is a separable subfactor $P$ of $\R^\u$ with $N\subseteq P$ such that $P'\cap \R^\u$ is a II$_1$ factor.
\end{fact*}

In \cite{jung}, we said the II$_1$ factor $M$ had the \textbf{Brown property} if, for all separable subfactors $N$ of $M^\u$, then there is a separable subfactor $P$ of $M^\u$ with $N\subseteq P$ such that $P'\cap M^\u$ is a II$_1$ factor.  It was shown in \cite{jung} that any $M\equiv \R$ has the Brown property.  In the last section of this paper, we prove a strengthening of this result:

\begin{thmC*}\label{4Brown}
If $M\equiv_4 \R$, then $M$ has the Brown property.
\end{thmC*}

% In this paper, we define a natural strengthening of the Brown property, called the \textbf{strong Brown property}, and prove the following result:

% \begin{thmC*}\label{strongbrown}
% The following statements are equivalent:
% \begin{enumerate}
%     \item $\R$ has the strong Brown property;
%     \item Every infinitely generic embeddable factor satisfies the FCEP.
% \end{enumerate}
% \end{thmC*}

An interesting question arises:  are these results actually improvements of their predecessors? Indeed, perhaps it is the case that there is $k\in \mathbb N$ such that if $M\equiv_k \R$, then $M\equiv \R$.  If this were to happen, then one would say that $\Th(\R)$ has \emph{quantifier simplification}.  Given recent results showing that the $\Th(\R)$ is very complicated from the model-theoretic perspective (see, e.g., \cite{ecfactor} and \cite{universaltheory}), we strongly believe in the following:

\begin{conj*}
$\Th(\R)$ does not admit quantifier simplification.
\end{conj*}

For the rest of this paper, we work under the assumption that the previous Conjecture has a positive solution.  In this case, Theorem A yields continuum many examples of factors satisfying the FCEP not covered by Theorem 1.  Similarly, Theorem C yields continuum many new examples of factors with the Brown property.

Infinitely generic embeddable factors form a subclass of the more general class of \textbf{existentially closed embeddable factors}.  An embeddable factor $M$ is existentially closed (e.c.) if:  whenever $N$ is an embeddable factor with $M\subseteq N$, there is an embedding $N\hookrightarrow M^\u$ that restricts to the diagonal embedding $M\hookrightarrow M^\u$.  It was noted in \cite{ecfactor} that $\R$ is an e.c. embeddable factor.  Existentially closed embeddable factors have proven very important in applications of model-theoretic ideas to the study of II$_1$ factors.  It is a major open question whether or not there are two non-elementarily equivalent e.c. embeddable factors.  If $\R$ is not infinitely generic, then we would have an example of such a pair of e.c. embeddable factors.  However, it could still be the case that all e.c. factors have the same 3-quantifier theory, in which case (2') in Theorem B is actually satisfied.

In order to keep this note relatively self-contained, we do not include much model-theoretic or operator-algebraic background.  A rather lengthy introduction to model-theoretic ideas as they pertain to problems around factorial commutants can be found in \cite{jung}.
% Now we consider Theorem C.  If $\R$ does not have the strong Brown property, then the FCEP is not true in general, which is of course very interesting!  Otherwise, let us suppose that $\R$ does have the strong Brown property, whence all infinitely generic embeddable factors satisfy the FCEP.  Is this covered by Theorems 1 or A above?  If it is covered by Theorem 1, then $\R$ is infinitely generic, which is indeed interesting as it fixes a hole in \cite{} and also shows that in Theorem 2 above, item (2), is satisfied, leaving just one very natural operator algebraic condition to check in order to know that all embeddable property (T) factors satisfy the FCEP.  Otherwise, $\R$ is not infinitely generic.  This is interesting as it presents the first example of two non-elementarily equivalent e.c. embeddable factors.  However, it could still be the case that the statement that all infinitely generic factors satisfy the FCEP because they all are 3-equivalent to $\R$.  In this case, item (2') in Theorem D is satisfied, whence once again we are happy.

In Section 2, we prove the main model-theoretic tools needed in the proof of Theorem A.  In Section 3 we prove Theorem A, in Section 4 we prove Theorem B, and in Section 5 we prove Theorem C.

\section{Weak heirs and weak embeddings}

In this section, we fix a continuous language $L$.  We say that a formula $\varphi$ is in \textbf{prenex normal form} if it is of the form $$Q_1x_1\cdots Q_m x_m \psi(x_1,\ldots,x_m,\vec y),$$ with each $Q_i\in \{\sup,\inf\}$ and with $\psi$ quantifier-free.  If the $Q_i$'s alternate type, then we say that $\varphi$ is $\forall_m$ (respectively $\exists_m$) if $Q_1=\sup$ (resp. $Q_1=\inf$).\footnote{Technically we really should be speaking of $m-1$ alternations of \emph{blocks} of quantifiers of the same length, but we blur this distinction here.}  If a formula is equivalent to a $\forall_m$ or $\exists_m$ formula, we often abuse terminology and refer to the formula itself as $\forall_m$ or $\exists_m$.

By a \textbf{fragment} of $L$-formulae, we mean a set $\Delta$ consisting of all $\forall_m$-formulae or of all $\exists_m$-formulae for some $m$.
% Let $\Delta$ and $\Delta'$ represent fragments of formulae, for us always of the form $Q_m$ for some $Q\in\{\forall,\exists\}$ and $m\geq 0$.

\begin{defn}
Fix an $L$-structure $M$, parameter sets $A\subseteq B\subseteq M$, and fragments $\Delta$ and $\Delta'$.
\begin{enumerate}
\item For $c\in M$, we set $\tp_\Delta^M(c/A)$ to be the set of all conditions $\varphi(x)=r$, where $\varphi\in \Delta$ has parameters from $A$ and $\varphi(c)^M=r$.
    \item $S_\Delta^M(A)$ denotes the set of all $\tp_{\Delta}^M(c/A)$ for $c\in M$.
    \item For $p\in S_\Delta^M(A)$ and $\varphi(x)$ a formula from $\Delta$ with parameters from $A$, we set $\varphi(x)^p$ to be the unique $r$ so that $\varphi(x)=r$ belongs to $p$.
    \item For $c\in M$, we set $\tp_{\Delta,\Delta'}^M(c/A,B)$ to be the union of $\tp_\Delta^M(c/A)$ and $\tp_{\Delta'}^M(c/B)$.
    \item We let $S_{\Delta,\Delta'}(A,B)$ denote the set of all $\tp_{\Delta,\Delta'}^M(c/A,B)$ for $c\in M$.  We extend the notation $\varphi(x)^p$ to $S_{\Delta,\Delta'}(A,B)$ in the obvious way.
    
    \item If $p\in S_\Delta(A)$, $q\in S_{\Delta,\Delta'}(A,B)$, and $\Delta'\subseteq \Delta$, we say that $q$ is an \textbf{heir} of $p$ if, for every $b\in B$, every $\varphi(x,y)\in \Delta'$, and every $\epsilon>0$, there is $a\in A$ such that $|\varphi(x,a)^p-\varphi(x,b)^q|<\epsilon$.
\end{enumerate}
\end{defn}

\begin{defn}
Suppose that $i:N\hookrightarrow M$ is an embedding between $L$-structures and $\Delta$ is a fragment.  We say that $i$ is:
\begin{enumerate}
    \item \textbf{downward $\Delta$} if, for any nonnegative formula $\varphi(x)\in \Delta$ and any $a\in N$, if $\varphi(i(a))^M=0$, then $\varphi(a)^N=0$;
    \item \textbf{upward $\Delta$} if, for any nonnegative formula $\varphi(x)\in \Delta$ and any $a\in N$, if $\varphi(a)^N=0$, then $\varphi(i(a))^M=0$.
\end{enumerate}
\end{defn}

We note one obvious fact:

\begin{lem}
Given an embedding $i:N\hookrightarrow M$, we have that $i$ is downwards $\exists_m$ if and only if $i$ is upwards $\forall_m$.
\end{lem}

\begin{proof}
Suppose that $i$ is not upwards $\forall_m$, so there is a nonnegative $\forall_m$ formula $\varphi(x)$ and $a\in N$ such that $\varphi(a)^N=0$ but $\varphi(i(a))^M=\epsilon>0$.  Then $(\epsilon\dminus \varphi(i(a)))^M=0$ and since this formula is equivalent to a $\exists_m$ formula, we have that $(\epsilon\dminus \varphi(a))^N=0$, a contradiction.  The other direction is similar.
\end{proof}

% In an earlier version, we had an upward assumption as well, but I think the previous lemma shows that this is obsolete.

The following is our main technical result concerning the existence of weak heirs.  In the remainder of this paper, $\u$ denotes a countably incomplete ultrafilter on some index set (unless otherwise specified).

\begin{thm}\label{heir}
Suppose that $M$ is a separable $L$-structure.  Fix a separable substructure $N$ of $M^\u$ such that the inclusion $N\subseteq M^\u$ is downward $\exists_{m+2}$.  Fix also $p\in S_{\forall_m}(N)$.  Then for any separable parameter set $A$ with $N\subseteq A\subseteq M^\u$ and any $n<m$, there is $q\in S_{\forall_m,\forall_n}(N,A)$ that is an heir of $p$.
% and a separable parameter set $A$ with $N\subseteq A\subseteq M^\u$.  Fix also $p\in S_{\forall_m}(N)$.  Suppose that the inclusion $N\subseteq M^\u$ is downward $\exists_{m+2}$ and upward $\forall_{m+1}$.  Then for any $n<m$, there is $q\in S_{\forall_m,\forall_n}(N,A)$ that is an heir of $p$.
\end{thm}

\begin{proof}
We seek $a\in M^\u$ satisfying the following two kinds of conditions:
\begin{enumerate}
\item $\psi(a)=\psi(x)^p$ for any $\forall_m$-formula $\psi(x)$ with parameters from $N$;
\item $\varphi(a,c)^{M^\u}\geq \frac{\epsilon}{2}$ for any $\forall_{n+1}$-formula $\varphi(x,y)$ with parameters from $A$ and any $\epsilon>0$ such that $\varphi(x,b)^p\geq \epsilon$ for all $b\in N$.
\end{enumerate}
Indeed, if $a$ is as above, we claim that $q:=\tp_{\forall_m,\forall_n}^{M^\u}(a/A)$ is an heir of $p$.  By (1), $q$ is an extension of $p$.  To see that $q$ is an heir, fix a $\forall_n$-formula $\varphi(x,c)$ with parameters from $A$ and set $s:=\varphi(x,c)^q=\varphi(a,c)^{M^\u}$.  Suppose, towards a contradiction, that there is $\epsilon>0$ such that $|\varphi(x,b)^p-s|\geq \epsilon$ for all $b\in N$.  It follows that $|\varphi(x,b)-s|^p\geq \epsilon$ for all $b\in N$.  Since $|\varphi(x,b)-s|$ is logically equivalent to a $\forall_{n+1}$, whence, by (2), $|\varphi(a,c)^{M^\u}-s|\geq \frac{\epsilon}{2}$, leading to a contradiction.

Suppose now, towards a contradiction, that no such $a\in M^\u$ exists.  By countable saturation, it follows that there are:
\begin{itemize}
\item a $\forall_m$-formula $\psi(x)$ with parameters from $N$ such that $\psi(x)^p=0$, 
\item a $\delta>0$, and 
\item formulae $\varphi_1(x,c_1),\ldots,\varphi_k(x,c_k)$ with parameters from $A$ as in (2)
\end{itemize}
such that, for any $a\in M^\u$, if $\psi(a)<\delta$, then $\varphi_i(a,c_i)<\frac{\epsilon}{2}$ for some $i=1,\ldots,k$. 

In other words, 
$$\left(\sup_x\min\left(\delta\dotminus \psi(x),\min_{1\leq i\leq k}\left(\varphi_i(x,c_i)\dotminus \frac{\epsilon}{2}\right)\right)\right)^{M^\u}=0.$$  Consequently, $$\left(\inf_{y_1}\cdots\inf_{y_k}\sup_x\min\left(\delta\dotminus \psi(x),\min_{1\leq i\leq k}\left(\varphi_i(x,y_i)\dotminus \frac{\epsilon}{2}\right)\right)\right)^{M^\u}=0,$$ and thus, since the inclusion $N\subseteq M^\u$ is downward $\exists_{m+2}$, we have 
$$\left(\inf_{y_1}\cdots\inf_{y_m}\sup_x\min\left(\delta\dotminus \psi(x),\min_{1\leq i\leq m}\left(\varphi_i(x,y_i)\dotminus \frac{\epsilon}{2}\right)\right)\right)^N=0.$$ Set $\eta:=\min(\delta,\frac{\epsilon}{2})$ and take $d_1,\ldots,d_k\in N$ such that $$\left(\sup_x\min\left(\delta\dotminus \psi(x),\min_{1\leq i\leq k}\left(\varphi_i(x,d_i)\dotminus \frac{\epsilon}{2}\right)\right)\right)^N<\eta;$$ since the inclusion $N\subseteq M^\u$ is upward $\forall_{m+1}$, we have $$\left(\sup_x\min\left(\delta\dotminus \psi(x),\min_{1\leq i\leq k}\left(\varphi_i(x,d_i)\dotminus \frac{\epsilon}{2}\right)\right)\right)^{M^\u}<\eta.$$  Take $a\in M^\u$ realizing $p$.  Then $\psi(a)^{M^\u}=\psi(x)^p=0$, whence, since $\eta\leq \delta$, we have $\min_{1\leq i\leq k}(\varphi_i(x,d_i)\dotminus \frac{\epsilon}{2})^{M^\u}<\eta\leq \frac{\epsilon}{2}$.  Choosing $i$ such that $(\varphi_i(a,d_i)\dotminus \frac{\epsilon}{2})^{M^\u}<\eta$, we get that $\varphi_i(x,d_i)^p=\varphi_i(a,d_i)^{M^\u}<\epsilon$, a contradiction.
\end{proof}

We will be interested in the following special case of Theorem \ref{heir}:

\begin{cor}\label{specialcase}
Suppose that $M$ is a separable $L$-structure.  Fix a separable substructure $N$ of $M^\u$ such that the inclusion $N\subseteq M^\u$ is downward $\exists_{3}$.  Fix also $p\in S_{\forall_1}(N)$.  Then for any separable parameter set $A$ with $N\subseteq A\subseteq M^\u$, there is $q\in S_{\forall_1,\forall_0}(N,A)$ that is an heir of $p$.
% Fix separable subsets $N$ and $A$ of $M^\u$ with $N\subseteq A\subseteq M^\u$ and $p\in S_{\forall_1}(N)$.  Suppose further that the inclusion $N\subseteq M^\u$ is downward $\exists_{3}$ and upward $\forall_{2}$.  Then there is $q\in S_{\forall_1,\operatorname{qf}}(N,A)$ that is an heir of $p$.
\end{cor}

\begin{defn}
Given a fragment $\Delta$ and an $L$-structure $M$, we set $$\Th_\Delta(M):=\{\sigma \ : \ \sigma \text{ is a nonnegative $L$-sentence from }\Delta \text{ and }\sigma^M=0\}.$$  If $N$ is another $L$-structure, we write $N\models \Th_\Delta(M)$ if $\sigma^N=0$ for all $\sigma\in \Th_\Delta(M)$.
\end{defn}

We now prove a result connecting small quantifier-fragments of theories of structures with the existence of embeddings as in the previous theorem.

\begin{prop}\label{fragmentsandembeddings}
Suppose that $M$ and $N$ are separable $L$-structures and $m\in \mathbb N$.  Then there is an embedding $i:N\hookrightarrow M^\u$ that is downwards $\exists_{m+2}$ if and only if $M\models \operatorname{Th}_{\exists_{m+3}}(N)$.
\end{prop}

\begin{proof}
First suppose that a downwards $\exists_{m+2}$-embedding $i:N\hookrightarrow M^\u$ exists and $\sigma$ is a nonnegative $\exists_{m+3}$-sentence such that $\sigma^N=0$.  Write $\sigma=\inf_x\varphi(x)$ with $\varphi$ a $\forall_{m+2}$-formula.  Fix $\epsilon>0$ and take $a\in N$ such that $\varphi(a)<\epsilon$.  Then $(\varphi(a)\dminus \epsilon)^N=0$, and since this formula is equivalent to a $\forall_{m+2}$-formula and $i$ is upwards $\forall_{m+2}$, we have that $(\varphi(i(a))\dminus \epsilon)^{M^\u}=0$.  Consequently, $(\inf_x(\varphi(x)\dminus\epsilon))^M=0$; since $M$ is arbitrary, we have that $\sigma^M=0$, as desired.

Conversely, suppose that $M\models \operatorname{Th}_{\exists_{m+3}}(N)$.  Let $L_N$ be the language obtained by adding constants $c_a$ for $a\in N$.  Set $\Gamma$ to be the following collection of $L_N$ sentences:
\begin{enumerate}
    \item $\theta(c_{a_1},\ldots,c_{a_n})$, where $\theta$ is a nonnegative quantifier-free formula and $\theta(a_1,\ldots,a_n)^N=0$;
    \item $\epsilon\dminus \varphi(c_{a_1},\ldots,c_{a_n})$, where $\varphi$ is a $\exists_{m+2}$-formula with $\varphi(a_1,\ldots,a_n)^N\geq \epsilon$
\end{enumerate}
If $\Gamma$ can be shown to be approximately finitely satisfiable in an expansion of $M$, then by countable saturation there is an expansion of $M^\u$ which is a model of $\Gamma$, and this yields the desired embedding.  So suppose $\theta_1,\ldots,\theta_k$ are as in (1) and $\epsilon_j\dminus \varphi_j$, $j=1,\ldots,l$, are as in (2).  Then $$\inf_x\left(\max\left(\max_{i=1,\ldots,k}\theta_i(x),\max_{j=1,\ldots,l}(\epsilon_j\dminus \varphi_j(x)\right)\right)$$ is equivalent to an $\exists_{m+3}$-sentence that evaluates to $0$ in $N$, whence, by assumption, also evaluates to $0$ in $M$.  This completes the proof.
\end{proof}

% \begin{prop}
% Suppose that $M\models \operatorname{Th}_{\exists_{m+3}}(N)$.  Then there is an embedding $N\hookrightarrow M^\u$ that is downwards $\exists_{m+2}$ and upwards $\forall_{m+1}$.
% \end{prop}

% I think that a converse holds:  If there is a downwards $\exists_{m+2}$ embedding $N\hookrightarrow M^\u$, then $M\models \operatorname{Th}_{\exists_{m+3}}(N)$.

Combining Theorem \ref{heir} and Proposition \ref{fragmentsandembeddings}, we arrive at:

\begin{cor}\label{whatwereallyuse}
Suppose that $M$ is a separable $L$-structure.  Fix a separable substructure $N$ of $M^\u$ such that $M\models \operatorname{Th}_{\exists_{m+3}}(N)$.  Fix also $p\in S_{\forall_m}^{M^\u}(N)$.  Then for any separable parameter set $A$ with $N\subseteq A\subseteq M^\u$ and any $n<m$, there is $q\in S_{\forall_m,\forall_n}^{M^\u}(N,A)$ that is an heir of $p$.  In particular, if $M\models \Th_{\exists_4}(N)$, then for any $p\in S_{\forall_1}^{M^\u}(N)$ and any separable parameter set $A$ with $N\subseteq A\subseteq M^\u$, there is $q\in S_{\forall_1,\forall_0}^{M^\u}(N,A)$ that is an heir of $p$.
% Fix separable subsets $N$ and $A$ of $M^\u$ with $N\subseteq A\subseteq M^\u$, $p\in S_{\forall m}(N)$, and $n< m$.  Further suppose that $M\models \operatorname{Th}_{\exists_{m+3}}(N)$.  Then there is $q\in S_{\forall_m,\forall_n}(N,A)$ that is an heir of $p$.
\end{cor}

\section{Proof of Theorem A}

In this section, we apply the abstract results from the previous section to the setting of II$_1$ factors.  Throughout this section, $L$ is the language of tracial von Neumann algebras and $T$ is the universal theory of embeddable tracial von Neumann algebras.  All structures considered in this section will be models of $T$.

\begin{lem}
Suppose that $M$ and $N$ are separable with $N\subseteq M^\u$.  Suppose also that $a,b\in M^\u$ are such that $a\in Z(N'\cap M^\u)$ and $\tp_{\forall_1}^{M^\u}(a/N)=\tp_{\forall_1}^{M^\u}(b/N)$.  Then $b\in Z(N'\cap M^\u)$.
\end{lem}

\begin{proof}
Since $\tp_{\forall_0}^{M^\u}(a/N)=\tp_{\forall_0}^{M^\u}(b/N)$, we have $b\in N'\cap M^\u$.  Now fix $\epsilon>0$.  By countable saturation, there are $e_1,\ldots,e_n\in N$ and $\delta>0$ such that, for all $c\in M^\u$, if $\|[c,e_i]\|_2<\delta$ for all $i=1,\ldots,n$, then $\|[c,a]\|_2<\epsilon$.  Consequently, 
$$\sup_x\min\left(\delta\dotminus min_i\|[x,e_i]\|_2, \|[x,y]\|_2\dotminus \epsilon\right)$$ belongs to $\tp_{\forall_1}^{M^\u}(a/N)$, whence it also belongs to $\tp_{\forall_1}^{M^\u}(b/N)$.  It follows that $b\in Z(N'\cap M^\u)$.  So, if $c\in N'\cap M^\u$, then $\|[b,c\|_2\leq \epsilon$.  Since $\epsilon$ was arbitrary, it follows that $[b,c]=0$, and thus $b\in Z(N'\cap M^\u)$, as desired.
\end{proof}

\begin{cor}\label{commutantcorollary}
Suppose that $N\subseteq P\subseteq M^\u$, $P'\cap M^\u$ is a factor, and every element of $S_{\forall_1}^{M^\u}(N)$ admits an heir to $S_{\forall_1,\forall_0}^{M^\u}(N,P)$.  Then $N'\cap M^\u$ is a factor.
\end{cor}

\begin{proof}
Take $a\in Z(N'\cap M^\u)$ and let $p:=\tp_{\forall_1}(a/N)$.  Let $q\in S_{\forall_1,\forall_0}(N,P)$ be an heir of $p$.  Let $b\in M^\u$ satisfy $q$.  By the heir property, $b\in P'\cap M^\u$.  If $c\in P'\cap M^\u$, then $c\in N'\cap M^\u$, whence, by the previous lemma, $[b,c]=0$.  It follows that $b\in Z(P'\cap M^\u)=\mathbb C$.  So $b=\lambda\cdot 1$ for some $\lambda \in \mathbb C$, so $d(x,\lambda\cdot 1)=0$ belongs to $q$, whence it also belongs to $p$, and thus $a=\lambda\cdot 1$, as desired.
\end{proof}

Recall the following fact of Nate Brown mentioned in the introduction:

\begin{fact}
For every separable $N\subseteq \R^\u$, there is a separable $P\subseteq \R^\u$ with $N\subseteq P$ such that $P'\cap \R^\u$ is a factor.
\end{fact}

We are now able to prove the following more precise version of Theorem A:

\begin{thm}
Suppose that $N$ is an embeddable factor such that $\R\models \operatorname{Th}_{\exists_4}(N)$.  Then $N$ satisfies the FCEP.
\end{thm}

\begin{proof}
Fix $P$ as in the previous fact, so $N\subseteq P\subseteq \R^\u$ with $P'\cap \R^\u$ a factor.  The proof then follows from Corollary \ref{whatwereallyuse} and Corollary \ref{commutantcorollary}.
\end{proof}

\section{Proof of Theorem B}

Let (*) denote the statement:  the amalgamated free product of embeddable factors over a property (T) base is once again embeddable.

\begin{lem}\label{star}
Suppose that (*) holds.  Then whenever $N$ is a w-spectral gap subfactor of the e.c. embeddable factor $M$, then $(N'\cap M)'\cap M=N$.
\end{lem}

\begin{proof}
In \cite{spectralgap}, this was proven without a restriction to embeddable factors.  The proof goes through in the embeddable case if one assumes (*) holds.
\end{proof}

Recall that if $N$ is a property (T) factor, then $N$ has a \textbf{Kazhdan set}, which is a finite subset $F$ of $N$ that satisfies the following property:  there is a $K > 0$ such that for any II$_1$ factor $M$ containing $N$ as a subfactor, any $b\in M_1$, and any sufficiently small $\eta>0$, if $\|[a,b]\|_2<\eta$ for all $a\in F$, then there is $c\in N'\cap M$ such that $\|b-c\|_2<K\eta$.  Since $\|b-E_{N'\cap M}(b)\|_2\leq \|b-c\|_2<K\eta$ and $E_{N'\cap M}$ is operator norm-contractive, it follows that we may assume that $c\in M_1$ as well.  (See \cite[Proposition 1]{CJ} for a proof.)

\begin{thm}\label{T}
Suppose that (*) holds.  Suppose further that $N$ is an embeddable property (T) II$_1$ factor, $M$ is an e.c. embeddable factor containing $N$, and $j:M\hookrightarrow \R^\u$ is downward $\Sigma_2$.  Then $j(N)'\cap \R^\u$ is a factor.
\end{thm}

\begin{proof}
Suppose, towards a contradiction, that $a\in Z(j(N)'\cap \R^\u)$ but $d(a,\tr(a)\cdot 1)=\epsilon>0$.  Without loss of generality, suppose $a$ is in the unit ball.  Let $\{z_1,\ldots,z_n\}$ be a Kazhdan set for $N$ with Kazhdan constant $K$.  Note that
$$\R^\u\models \forall w\left(\max_{1\leq i\leq n}\|[w,j(z_i)]\|_2=0\rightarrow \|[w,a]\|_2=0\right),
$$ whence, by \cite[Proposition 7.14]{mtfms}, there is a continuous, nondecreasing function $\alpha:\mathbb{R}\to \mathbb R$ satisfying $\alpha(0)=0$ such that
$$\R^\u\models \sup_w\left(\|[a,w]\|_2\dotminus \alpha\left(\max_{1\leq i\leq n}\|[w,j(z_i)]\|_2\right)\right)=0.$$ Set $\psi(x,\vec t):=\sup_w(\|[x,w]\|_2\dotminus \alpha(\max_{1\leq i\leq n}\|[w,t_i]\|_2))$, a universal formula such that $\R^\u\models \psi(a,j(\vec z))=0$ whence
$$\R^\u\models \inf_x\max\left(\max_{1\leq i\leq n}\|[x,j(z_i)]\|_2, \psi(x,j(\vec z)),\epsilon\dotminus d(x,tr(x)\cdot 1)\right)=0.$$ Since the latter displayed formula is equivalent to a $\exists_2$-formula, by assumption we have $$M\models \inf_x\max\left(\max_{1\leq i\leq n}\|[x,z_i]\|_2, \psi(x,\vec z),\epsilon\dotminus d(x,tr(x)\cdot 1)\right)=0.$$  Fix $\eta>0$ sufficiently small and take $b\in M_1$ such that
$$M\models \max\left(\max_{1\leq i\leq n}\|[b,z_i]\|_2, \psi(b,\vec z),\epsilon\dotminus d(b,tr(b)\cdot 1)\right)<\eta.$$  If $\eta$ is sufficiently small, there is $b'\in N'\cap M$ such that $d(b,b')<K\eta$.  For simplicity, set $\beta:=K\eta$.  Now suppose that $c\in N'\cap M$ is in the unit ball.  Then $\|[b,c]\|_2<\eta$, whence $\|[b',c]\|_2<\eta+2\beta$.  Since $c\in N'\cap M$ was arbitrary, we have $d(b',(N'\cap M)'\cap M)\leq \eta+2\beta$.\footnote{This follows from the general fact that, for a subfactor $P$ of a II$_1$ factor $Q$ and $a\in Q_1$, one has $d(a,P'\cap Q)\leq \sup_{b\in P_1}\|[a,b]\|_2$.}  By Lemma \ref{star}, since $M$ is e.c. and $N$ has w-spectral gap in $M$, we have that $(N'\cap M)'\cap M=N$, so $d(b',N)\leq \eta+2\beta$, that is, $d(b',E_N(b'))\leq \eta+2\beta$.  However, $b'\in N'\cap M$ implies $E_N(b')\in Z(N)=\mathbb C$.  It follows that $d(b',\tr(b')\cdot 1)=d(b,\mathbb C)\leq d(b,E_N(b'))\leq \eta+2\beta$.  Since $\epsilon\dotminus d(b,\tr(b)\cdot 1)<\eta$, we have that $\epsilon\dotminus d(b',\tr(b')\cdot 1)<\eta+2d(b,b')<\eta+2\beta$, which is a contradiction as long as $2\eta+4\beta<\epsilon$.  Recalling that $\beta=K\eta$, we have that $2\eta+4\beta=(2+4K)\eta$, whence choosing $\eta<\frac{\epsilon}{2+4K}$, we arrive at the desired contradiction.
% Take $b''\in N$ with $d(b',b'')\leq \eta+2\beta$.  Thus, $d(b',b'')\leq \eta+3\beta$.  Now recall also that $\epsilon\dotminus d(b,\tr(b)\cdot 1)<\eta$; since $d(b,b'')<\eta+3\beta$, we have that $\epsilon\dotminus d(b'',\tr(b'')\cdot 1)<3\eta+6\beta$.  However, since $\|[w,b]\|_2<\eta$, we have that $\|[w,b'']\|_2<2\eta+3\beta$, so if $2\eta+3\beta$ is small enough, then each $\|[z_i,b'']\|_2$ is small enough so that there is $b'''\in Z(N)=\mathbb C$ with $d(b'',b''')<\frac{\epsilon}{2}$.  Thus, $d(b'',\tr(b'')\cdot 1)<\frac{\epsilon}{2}$.  If $3\eta+6\beta<\frac{\epsilon}{2}$, then this gives a contradiction.
\end{proof}

The following is a more precise version of Theorem D; it follows immediately from Proposition \ref{fragmentsandembeddings} and Theorem \ref{T}.

\begin{cor}\label{preciseD}
Suppose that (*) holds and every embeddable factor $N$ embeds into an e.c. embeddable factor $M$ such that $M\models \Th_{\forall_3}(\R)$.  Then every embeddable property (T) factor satisfies the FCEP.
\end{cor}

The assumption in the previous corollary should be compared to:

\begin{lem}
If $M$ is an e.c. embeddable factor, then $M\models \operatorname{Th}_{\exists_3}(\R)$.
\end{lem}

\begin{proof}
Since $M$ is a II$_1$ factor, we may assume that $\R\subseteq M$.  Fix an $\exists_3$-sentence $\sigma=\inf_x\sup_y\inf_z \varphi(x,y,z)$ such that $\sigma^\R=0$.  Fix $\epsilon>0$ and $a\in \R$ such that $(\inf_y\sup_z\varphi(a,y,z))^\R<\epsilon$.  Fix $b\in M$ and an embedding $i:M\hookrightarrow \R^\u$.  Then $(\inf_z\varphi(i(a),i(b),z)^{\R^\u}<\epsilon$, whence there is $c\in \R^\u$ such that $(\varphi(i(a),i(b),c)^{\R^\u}<\epsilon$.  Since $M$ is e.c. there is $b'\in M$ such that $\varphi(a,b,c')<2\epsilon$.  Since $\epsilon$ is arbitrary, we have that $\sigma^M=0$.
\end{proof}

Thus, the assumption of Corollary \ref{preciseD} comes tantalizingly close to removing any model-theoretic assumption at all, leaving only the operator-algebraic assumption (*).

% This is tantalizingly close since $M\models \operatorname{Th}_{\exists_3}(\R)$ for any e.c. embeddable factor $M$.  Also this works for any e.c. embeddable $M$ and we don't have to worry about trying to prove something about infintiely generic embeddable factors.

\section{Proof of Theorem C}

We begin by explaining exactly what we mean for two structures to be $k$-elementarily equivalent.
\begin{defn}
If $\varphi$ is a formula and $k$ is a nonnegative integer, we recall what it means for $\varphi$ to have \textbf{quantifier depth at most $k$}, written $\depth(\varphi)\leq k$, by induction on the complexity of $\varphi$:
\begin{itemize}
\item If $\varphi$ is atomic, then $\depth(\varphi)\leq 0$.
\item If $\varphi_1,\ldots,\varphi_n$ are formulae, $f:\mathbb R^n\to \mathbb R$ is a continuous function and $\varphi=f(\varphi_1,\ldots,\varphi_n)$, then $\depth(\varphi)\leq\max_{1\leq i\leq n}\depth(\varphi_i)$.
\item If $\varphi=\sup_{\vec x} \psi$ or $\varphi=\inf_{\vec x} \psi$, then $\depth(\varphi)\leq \depth(\psi)+1$.
\end{itemize}
\end{defn}

\begin{defn}
If $M$ and $N$ are $L$-structures, we write $M\equiv_k N$ if $\sigma^M=\sigma^N$ whenever $\operatorname{depth}(\sigma)\leq k$.
\end{defn}

\begin{remark}
If $\sigma$ is an $\forall_m$-sentence or a $\exists_m$-sentence, then clearly $\depth(\sigma)=m$.  Consequently, if $M\equiv_m N$, then $M\models \Th_{\forall_m}(N)$ and $N\models \Th_{\forall_m}(M)$.
\end{remark}

We recall the following Ehrenfeucht-Fraisse game for continuous logic.
\begin{defn}
Let $M$ and $N$ be $L$-structures and let $k\in \mathbb N$.  $\mathfrak G(M,N,k)$ denotes the following game played by two players.  First, player I plays either a tuple\footnote{Here, tuples can be either of finite or countably infinite length.} $\vec{x_1}\in M$ or a tuple $\vec{y_1}\in N$.  Player II then responds with a tuple $\vec{y_1}\in N$ or $\vec{x_1}\in M$.  The play continues in this way for $k$ rounds.  We say that \emph{Player II wins $\mathfrak G(M,N,k)$} if there is an isomorphism between the substructures generated by $\{\vec{x_1},\ldots,\vec{x_k}\}$ and $\{\vec{y_1},\ldots,\vec{y_k}\}$ that maps $\vec{x_i}$ to $\vec{y_i}$.  
\end{defn}

\begin{defn}
If $M$ and $N$ are $L$-structures, we write $M\equiv_k^{EF} N$ if II has a winning strategy for $\mathfrak G(M,N,k)$. 
\end{defn}

It is a routine induction to show that $M\equiv_k^{EF} N$ implies $M\equiv_k N$.  Conversely, one has the following result (see \cite[Lemma 2.4]{gamespaper}):  

\begin{fact}
Suppose that $M$ and $N$ are countably saturated $L$-structures.  Then $M\equiv_k N$ if and only if $M\equiv_k^{EF} N$.
\end{fact}

We are now ready to prove Theorem C.  Recall from the introduction that a II$_1$ factor $M$ has the Brown property if:  for every separable subfactor $N$ of $M^\u$, there is a separable subfactor $P$ of $M^\u$ with $N\subseteq P$ such that $P'\cap M^\u$ is a II$_1$ factor.

\begin{thm}
Suppose that $M\equiv_4 \R$.  Then $M$ has the Brown property.
\end{thm}

\begin{proof}
Suppose $N$ is a separable subfactor of $M^\u$.  It suffices to find a separable subfactor $P$ of $M^\u$ containing $N$ such that $P'\cap M^\u$ is a factor.  Indeed, since $M\equiv_2 \R$, $M$ is McDuff, whence $P'\cap M^\u$ will contain a copy of $\R^\u$ and will thus be a II$_1$ factor, as desired.

Since $M\equiv_4 \R$ and $M^\u$ and $\R^\u$ are $\aleph_1$-saturated, we know that player II has a winning strategy in $\mathfrak G(M^\u,\R^\u,4)$.  We assume in the following run of the game that player II plays according to this strategy.  Let player I begin with $\vec a_1$, which is a countable sequence from the unit ball of $N$ which generates $N$.  Let player II respond with $\vec b_1$ and let $N^*$ denote the separable subfactor of $\R^\u$ generated by $\vec b_1$.  Since $\R$ has the Brown property, there is a separable subfactor $P^*$ of $\R^\u$ containing $N^*$ such that $(N^*)'\cap \R^\u$ is a factor.  Let $\vec b_2$ be a countable subset of the unit ball of $P^*$ which, together with $\vec b_1$, generates $P^*$.  Let player II respond with $\vec a_2$ and let $P$ be the separable subfactor of $M^\u$ generated by $\vec a_1$ and $\vec a_2$.  We claim that this $P$ is as desired.

To see this, suppose that $a_3\in Z(P'\cap M^\u)$.  We wish to show that $a_3\in \mathbb C$.  To see this, let player II respond with $b_3\in \R^\u$.  We claim that $b_3\in Z((P^*)'\cap \R^\u)$, whence $b_3\in \mathbb C$.  To see this, suppose that $b_4\in (P^*)'\cap \R^\u$.  Let player II respond with $a_4\in M^\u$.  Since the map $\vec a_1\vec a_2 a_3 a_4\mapsto \vec b_1\vec b_2 b_3 b_4$ extends to an isomorphism between the subalgebras they generate, we see that $a_4\in P'\cap M^\u$.  It follows that $a_3$ and $a_4$ commute, whence so do $b_3$ and $b_4$.

Now that we have established that $b_3\in \mathbb C$, the fact that the strategy is winning also shows that $a_3\in \mathbb C$, as desired.
\end{proof}

Recall that a McDuff II$_1$ factor is \textbf{super McDuff} if $M'\cap M^\u$ is a II$_1$ factor.  In \cite[Proposition 4.2.4]{jung}, it was proven that $M$ has the Brown property if and only if all $N$ elementarily equivalent to $M$ are super McDuff.  Consequently, we arrive at:

\begin{cor}
If $M\equiv_4 \R$, then $M$ is super McDuff.
\end{cor}

As mentioned in the introduction, if $\Th(\R)$ does not admit quantifier simplification, then these results yield continuum many new examples of separable factors that are super McDuff and have the Brown property.

\end{document}